\documentclass[11pt]{amsart}

\usepackage{epigamath}

\usepackage[english]{babel}

\usepackage{enumitem}
\usepackage[all]{xy}

\newtheorem{thm}{Theorem}[section]
\newtheorem{lem}[thm]{Lemma}

\newtheorem{prop}[thm]{Proposition}

\newtheorem{conj}[thm]{Conjecture}
\newtheorem*{thmA}{Theorem A}
\newtheorem*{thmB}{Theorem B}
\newtheorem*{WConj}{Weak Conjecture}
\newtheorem*{SConj}{Strong Conjecture}
\theoremstyle{definition}

\newtheorem{claim}[thm]{Claim}
\newtheorem{rmk}[thm]{Remark}

\newcommand{\ga}[2]{\begin{gather}\label{#1}#2 \end{gather}}

\newcommand{\Hom}{{\rm Hom}}

\newcommand{\Spec}{{\rm Spec \,}}

\newcommand{\sA}{{\mathcal A}}

\newcommand{\sC}{{\mathcal C}}

\newcommand{\sE}{{\mathcal E}}
\newcommand{\sF}{{\mathcal F}}

\newcommand{\sL}{{\mathcal L}}

\newcommand{\sO}{{\mathcal O}}

\newcommand{\sS}{{\mathcal S}}

\newcommand{\A}{{\mathbb A}}

\newcommand{\C}{{\mathbb C}}

\newcommand{\F}{{\mathbb F}}

\renewcommand{\P}{{\mathbb P}}
\newcommand{\Q}{{\mathbb Q}}

\newcommand{\Z}{{\mathbb Z}}

\DeclareMathOperator{\Ind}{Ind}

\DeclareMathOperator{\Ch}{char}

\DeclareMathOperator{\Spm}{Spm}

\DeclareMathOperator{\Spf}{Spf}

\DeclareMathOperator{\GL}{GL}

\DeclareMathOperator{\PD}{PD}
\DeclareMathOperator{\Def}{D}

\EpigaVolumeYear{6}{2022} \EpigaArticleNr{5} \ReceivedOn{June 16, 2020}
\InFinalFormOn{October 31, 2021}
\AcceptedOn{December 20, 2021}

\title{Density of arithmetic representations of function fields}
\titlemark{Density of arithmetic representations of function fields}

\author{H\'el\`ene Esnault}
\address{Freie Universit\"at Berlin, Arnimallee 3, 14195, Berlin,  Germany; and \\
The Institute for Advanced Study, Mathematics, 1 Einstein Dr., Princeton, NJ 08540, USA}
\email{esnault@math.fu-berlin.de and esnault@ias.edu}
\author{Moritz Kerz}
\address{Fakult\"at f\"ur Mathematik,
Universit\"at Regensburg,
93040 Regensburg, Germany}
\email{moritz.kerz@mathematik.uni-regensburg.de}

\authormark{H. Esnault and M. Kerz}

\AbstractInEnglish{We propose a conjecture on the density of arithmetic points in the deformation space of
  representations of the \'etale fundamental group in positive characteristic. This
  conjecture has applications   to \'etale cohomology theory, for example it implies a Hard
  Lefschetz conjecture. We prove
  the density conjecture in tame degree two for the curve $\P^1\setminus \{0,1,\infty\}$.}

\MSCclass{11G99; 14G99}

\KeyWords{Arithmetic  local systems; deformation
space; density conjecture}

\acknowledgement{The first author is supported by the Institute for Advanced Study, Princeton, the second author by the SFB 1085 Higher Invariants, Universit\"at Regensburg}

\begin{document}

%%%%%%%%%%%%%%%%%%%%%%%%%%%%%%%
% Title page
%%%%%%%%%%%%%%%%%%%%%%%%%%%%%%%

%\removeabove{}
%\removebetween{}
%\removebelow{}

\maketitle

\begin{prelims}

\DisplayAbstractInEnglish

\bigskip

\DisplayKeyWords

\medskip

\DisplayMSCclass

%\bigskip

%\languagesection{Fran\c{c}ais}

%\bigskip

%\DisplayTitleInFrench

%\medskip

%\DisplayAbstractInFrench

\end{prelims}

%%%%%%%%%%%%%%%%%%%%%
% Table of Contents
%%%%%%%%%%%%%%%%%%%%%

\newpage

\setcounter{tocdepth}{1}

\tableofcontents

%%%%%%%%%%%%%%%%%%%%%
% Content begins here
%%%%%%%%%%%%%%%%%%%%%

\section{Introduction}

Let $X_0$ be a smooth geometrically connected variety   defined over a finite field $k=\F_q$
of characteristic $p$. We fix  an algebraic closure $k\subset \bar k$ and a
geometric point $x\in X_0(\bar k) $. In this note we  study representations of the geometric \'etale fundamental group
$G=\pi_1^{\rm \acute{e}t}(X,x)$, where $X=X_0\otimes_k \bar k$, and the action of the
Frobenius on the set of  representations.

For a given prime $\ell\neq p$, we fix a finite field $\F$ of characteristic $\ell$,
and a continuous semi-simple representation $\bar \rho: G\to {\GL}_r(\F)$.
 We define the set  $ \mathcal S_{\bar \rho}$
  of isomorphism classes of continuous semi-simple representations $ \rho\colon \pi_1^{\rm
    \acute{e}t}(X, x)\to {\GL}_r(\bar \Q_\ell)$ with the property that the associated semi-simple residual
  representation is isomorphic to $\bar \rho$.
  We endow $ \mathcal S_{\bar \rho}$ with a Noetherian Zariski topology in Section~\ref{sec:zartop}.

There is a canonical  Frobenius action $\Phi\colon \sS_{\bar \rho}\xrightarrow{\sim}
\sS_{\bar \rho}$.
 A point  $[\rho]\in \sS_{\bar \rho}$ is fixed by $\Phi^n$ for some  integer $n>0$ if and
 only if the representation $\rho$ extends to a continuous representation $\pi_1^{\rm
   \acute{e}t}(X_0\otimes_k k' ,x)\to \GL_r(\bar\Q_\ell )$, for some finite extension
 $k\subset k'$. We call such a point in $\sS_{\bar\rho}$
 arithmetic and we let $\sA_{\bar\rho}\subset\sS_{\bar\rho}$ be the subset of arithmetic points.

 The aim of our note is to propose and to study (two variants of) a conjecture about the
 density of arithmetic points, see Section~\ref{sec:conj}.

\begin{WConj} The arithmetic points $\sA_{\bar\rho}$ are dense in $\sS_{\bar \rho}$.
\end{WConj}

\begin{SConj}
 For a Zariski closed subset $Z\subset \sS_{\bar \rho}$ with $\Phi^n(Z)=Z$ for some
 integer $n>0$
the subset of arithmetic points $Z\cap \sA_{\bar\rho} $ is dense in $Z$.
\end{SConj}

One application of the Strong Conjecture is that it implies a Hard Lefschetz isomorphism
for semi-simple perverse $\bar\Q_\ell$-sheaves in characteristic~$p$, see
Section~\ref{sec:applications}. This application is motivated by the corresponding work of
Drinfeld for complex varieties~\cite{Dri}.

For degree $r=1$ and $X$ either proper or a torus the Strong Conjecture is shown in~\cite[Theorem~1.7 and Lemma~3.1]{EK19}.

\smallskip

 In Section~\ref{sec:resind} we prove the following reductions for the Strong Conjecture,
 see also
 Proposition~\ref{prop:lefbelyi}. Here the algebraically closed field $\bar k$ is fixed.

 \begin{itemize}
 \item
   If the Strong Conjecture holds  for given degree~$r$  for all smooth curves $X$ over
   $\bar k$ then
   it holds in degree~$r$  for all smooth varieties $X$ over $\bar k$.
   \item If the
Strong Conjecture holds
in any degree $r$ for   $X=\P^1_{\bar k}\setminus \{0,1,\infty\}$ and $\bar \rho$  tame then it
holds in general over $\bar k$.
\end{itemize}

These reductions motivate our two main theorems, see Section~\ref{sec:conj}.

\begin{thmA}  The
 Weak Conjecture holds  when $X$ is a curve, $\ell > 2$ and $\bar \rho$ is absolutely irreducible.
\end{thmA}

\begin{thmB} The
Strong Conjecture holds for $X=\P^1_{\bar k}\setminus \{0,1,\infty\}$ when $\bar \rho$ is tame of degree two.
\end{thmB}

 We now explain the ideas of our proofs.  The main ingredient in the proof of  Theorem~A
 is de Jong's conjecture~\cite{deJ01}  proven in \cite[\S~1.4]{Gai07}  under the assumption
 $\ell > 2$, using the geometric Langlands program. Indeed, if $\bar \rho$ is absolutely irreducible, then $\sS_{\bar \rho}$ is the set of $\bar \Q_\ell$-points of Mazur's deformation  space which is smooth if $X$ is a curve, and on which we can apply de Jong's technique  \cite[\S~3.14]{deJ01}.

The proof of Theorem~B is very different. We embed $\sS_{\bar \rho}$   in the completion
of the affine space of dimension $6$ at the closed point which corresponds to the
characteristic polynomials  of three well chosen elements of the geometric fundamental
group $G$
on which $\Phi$ acts by raising to the  $q^\mathrm{th}$ power.
 We can then apply  our main density theorem in~\cite{EK19} on the cover which separates the roots of those polynomials.
 In particular, this also shows that the arithmetic points are precisely those which have quasi-unipotent monodromy at infinity.
We remark in Section~\ref{ss:dJ} that our method
yields a proof de Jong's conjecture  in this particular case, which does not use
automorphic forms.

\bigskip

\subsection*{Acknowledgements} We thank Daniel Litt for discussions around the topic of our note,
 Michel Brion for a discussion on invariants,  Ga\"etan Chenevier and Gebhard B\"ockle for a discussion on induced determinants.
 We thank Akshay Venkatesh and Mark Kisin for the Remark~\ref{rmk:akshaymark}.  We thank the referee for the friendly and helpful report.

\section{The Zariski topology on the set of semi-simple
  representations} \label{sec:zartop}

Let $\ell$ be a prime number, $\sO$ be the ring of integers of a finite extension of $\Q_\ell$ with residue field $\F$,
 $\sO\hookrightarrow \bar \Q_\ell$ be an embedding  of $\sO$ into an algebraic closure of $\bar \Q_\ell$  defining an embedding of $\F$ into an algebraic closure  $
\bar \F$.
Let $G$ be a pro-finite group which satisfies Mazur's $\ell$-finiteness property, {\it i.e.}\
for any open subgroup $U\subset G$ the set $\Hom_{\rm cont}(U, \mathbb Z/\ell \mathbb Z)$
is finite. Let
\ga{}{ \bar \rho:G\to \GL_r(\F)  \notag}
be a continuous representation.
We define  $ \mathcal S_{\bar \rho}$
 to be the {\it  set of isomorphism classes of continuous semi-simple representations}
 $\rho\colon G\to \GL_r(\bar \Q_\ell)$ {\it with semi-simple reduction isomorphic to }
 $\bar \rho^{\rm ss}: G\to \GL_r(\F)\subset \GL_r(\bar \F)$.
In this section we define a Zariski topology on $ \mathcal S_{\bar \rho}$. In
Section~\ref{sec:defrings} we relate $\sS_{\bar\rho}$ to the deformation space of pseudorepresentations.

\medskip

For a finite family $\underline g = (g_1, \ldots ,  g_m)\in G^m$,  let $p_i$ be the
characteristic polynomial  ${\Ch}(\bar \rho(g_i))$ of $\bar \rho (g_i)$. Then $\underline
p= (p_1, \ldots , p_m)$ is an $\F$-point of the affine space $\A^{rm}_\sO$ over $\sO$.
Let $R_{\underline p}=R_{\bar \rho(\underline{g})}$ be the complete local ring of $\A^{rm}_{\sO}$ at the
closed point $\underline p$. The ring  $R_{\bar \rho(\underline{g})}
\otimes_\sO \bar \Q_\ell $ is  Noetherian Jacobson.
Its maximal ideals  correspond to the $m$-tuples of polynomials over $\bar
\Z_\ell$ with reduction $\underline p$, see~\cite[Propositions~A.2.2.2 and~A.2.2.3]{GL96}.
Sending a representation   $\rho\colon G\to \GL_r(\bar \Q_\ell)$ in $\sS_{\bar \rho}$ to the family of characteristic polynomials
$(\Ch (\rho(g_1)), \ldots ,\Ch (\rho(g_m) ))$ therefore
induces a map
\[
  \Ch_{\underline g} \colon \sS_{\bar \rho}\to \Spm (R_{\bar
  \rho(\underline{g})} \otimes_\sO \bar \Q_\ell ).
\]
We endow the maximal spectrum with the usual Zariski topology, which is thus Noetherian.
\begin{prop}\label{prop.Zartop}
  There  exists an integer $\tilde m >0$ and  a  family $\underline{\tilde g}\in G^{\tilde m}$ such that for
  any finite family $\underline{g}\in G^{ m}$ which contains $\underline{\tilde g} $ we have:
\begin{itemize}
\item[\rm (1)]
$\Ch_{\underline g} $ is injective with Zariski closed image.
\item[\rm (2)]
The induced topologies on $\sS_{\bar\rho}$ via the embeddings $\Ch_{\underline g} $ and
$\Ch_{\underline{\tilde g}} $ are the same.
\end{itemize}
\end{prop}

Proposition~\ref{prop.Zartop} is  an immediate consequence of  Lemma~\ref{lem:fin}.
{\it  From now on we  endow}  $\sS_{\bar\rho}$ {\it with the induced Zariski topology from Proposition}~\ref{prop.Zartop}.

\begin{rmk}\label{rmk:l-adictop}
By the same procedure we can define the $\ell$-adic topology on  $\sS_{\bar\rho}$, which
we do not consider in this note, compare~\cite[Theorem~D]{Che14}, and \cite{Lit19} where it is used in an essential way.
\end{rmk}

\section{The density conjectures} \label{sec:conj}
In this section we   formulate  a strong conjecture and a weak one  on the density of arithmetic representations
in the Zariski space of all semi-simple representations $\sS_{\bar \rho}$ defined in Section~\ref{sec:zartop}. Then we formulate our main results
concerning them.

\smallskip

Let $X_0$ be a smooth geometrically connected variety defined over a finite field $k=\F_q$
of characteristic $p\neq \ell$.  Set $X= X_0\otimes_k \bar k$, where $\bar k$ is an
algebraic closure of $k$.  Fix a geometric point $x\in X_0( \bar k)$  and let $G$ be the
geometric fundamental group $\pi_1^{\rm \acute{e}t}(X,x)$.  Fix  a
lift $\Phi\in \pi_1^{\rm \acute{e}t}(X_0,x)$ of the {\it arithmetic Frobenius}.  Then $\Phi$ acts by conjugation on $G$.
This action depends on the lift up to an inner automorphism, so it canonically
acts on isomorphism classes of representations of $G$.  We assume that
\ga{}{\Phi(\bar \rho) \ \text{ is isomorphic to }\ \bar \rho\notag}
which is always fulfilled after replacing $\Phi$ by a power, or equivalently $X_0$ by $X_0\otimes_k k'$ for a finite extension $k'$ of $k$.
Thus the action of $\Phi$ on $G$ induces a well defined automorphism $\Phi$ of
$\sS_{\bar \rho}$.
By the construction of the Zariski topology on  $\sS_{\bar \rho}$ via
 Proposition~\ref{prop.Zartop} the automorphism $\Phi$ is a homeomorphism.

 We define the {\it arithmetic points} of $\sS_{\bar \rho}$ as the fixed points of powers
of $\Phi$
\ga{}{\sA_{\bar \rho}:= \bigcup_{n>0} \sS_{\bar \rho}^{\Phi^n}.\notag}

\begin{rmk}\label{rmk:arithext}
The arithmetic points in  $\sS_{\bar \rho}$ correspond to those continuous semi-simple representations
$\rho\colon G\to \GL_r(\bar \Q_\ell)$  which can be extended to a continuous representation
\[
  \pi_1^{\rm \acute{e}t}(X_0\otimes_k k',x) \to \GL_r(\bar \Q_\ell)
\]
for some finite extension $k'\subset \bar k$ of $k$,  see~\cite[\S~1.1.14]{Del80}.
\end{rmk}

\begin{conj}[Weak Conjecture] \label{conj:weak}
The space $\sS_{\bar \rho}$ is the Zariski closure of its arithmetic points
 $\sA_{\bar \rho}$.
\end{conj}

\begin{conj}[Strong Conjecture] \label{conj:strong}
A  Zariski closed subset $Z\subset \sS_{\bar \rho}$  with $\Phi^n(Z)=Z$ for some integer $n>0$ is the Zariski closure of its arithmetic points
$Z\cap \sA_{\bar \rho}$.
\end{conj}

Note that the formulation of the  conjectures depends only on $X$ and not on the choice of $X_0$ or the base
point $x$.

\begin{rmk} If $r=1$ and $X$ is projective or $X$ is a torus, then the strong conjecture is true by virtue of
\cite[Theorem~1.7 and Lemma~3.1]{EK19}.
\end{rmk}
\begin{rmk}
If we endow $\sS_{\bar \rho}$ with the $\ell$-adic topology as in
Remark~\ref{rmk:l-adictop},  then the subset of arithmetic points $\sA_{\bar \rho}$ is
discrete and closed, see~\cite[Theorem~1.1.3]{Lit19}.
\end{rmk}
Using the Lefschetz theorem on fundamental groups and the  Belyi principle we reduce in Section~\ref{sec:resind}  the Strong Conjecture
to the case where $X$ is a curve.
\begin{prop}\label{prop:lefbelyi}
  For varieties over  the fixed field $\bar k$ we have the implications:
  \begin{itemize}
    \item[\rm (1)]
If for fixed $r$ the Strong Conjecture holds for $\dim(X)=1$, then it holds for any $X$
and the given degree~$r$.
\item[\rm (2)] If the Strong conjecture holds for all $r>0$ for tame representations $\bar
  \rho$ on the variety
  $X=\P^1_{\bar k}\setminus \{0,1, \infty\}$, then it holds in general.
  \end{itemize}
\end{prop}
The main results of our note are  the following.
\begin{thm}[Theorem A]\label{thmA} Assume that $\ell>2$.
If $\bar \rho$ is absolutely irreducible and $X$ is a curve, then the weak conjecture holds.
\end{thm}
\begin{thm}[Theorem B]\label{thmB}
If $X=\P^1_{\bar k}\setminus \{0,1, \infty\}$, $r=2$ and $\bar \rho$ is tame, then the strong conjecture holds.  The arithmetic local systems are then precisely those with  quasi-unipotent monodromies at infinity.
\end{thm}
The only reason why we assume $\ell>2$ in Theorem~\ref{thmA} is that de
Jong's conjecture \cite[Conjecture~2.3]{deJ01} is  known only under this assumption at the moment, see~\cite[\S~1.4]{Gai07}.
In fact our proof of Theorem~\ref{thmB} yields a geometric proof of de Jong's conjecture on $\P^1_k\setminus \{0,1, \infty\}$ in rank $2$ for any $\ell$, without any use of the
Langlands  program, see Section~\ref{ss:dJ}.

\section{The deformation space  of pseudorepresentations}\label{sec:defrings}

In this section we recall some properties of the deformation space of pseudorepresentations $\PD_{\bar
\rho}$ following~\cite{Che14}. The
reason why we work with pseudorepresentations is that they naturally give rise to a
parametrization of the semi-simple representations $\sS_{\bar \rho}$ defined in Section~\ref{sec:zartop}.  As in Section~\ref{sec:zartop}, $G$ is a profinite group satisfying Mazur's $\ell$-finiteness property. 

\smallskip

Let $\sC$ be the category of
complete local $\sO$-algebras $(A, \mathfrak{m}_A)$ such that $\sO\to
A/\mathfrak{m}_A$ identifies the residue fields of $A$ and $\sO$.
Following \cite[Section~3]{Che14}  we define the {\it functor of pseudodeformations} of $\bar \rho$
\ga{}{ \PD_{\bar \rho}: \sC\to {\rm Sets} \notag}
which assigns  to $A$ the set  of continuous $r$-dimensional  $A$-valued determinants $D\colon A[G]\to A$  such that $D\otimes_A \F\colon \F[G]\to \F$ is the
$\F$-valued determinant induced by~$\bar \rho$.

Recall that a determinant is given by a compatible collection of maps $D_B \colon B[G]\to
B $, where $B$ runs through all commutative $A$-algebras, see~\cite[\S~1.2]{Che14}. Every
continuous representation $\rho\colon G\to \GL_r(A)$ gives rise to a continuous determinant
${\rm Det}(\rho )\colon A[G]\to A$. We  define the
{\em coefficients of the characteristic polynomial} of $D$ as the maps $\Lambda_i \colon G
\to A$ determined by the formula
\[
  D_{A[t]}(t-[g])=  \sum_{i=0}^r (-1)^i \Lambda_i(g) t^{d-i},
\]
see~\cite[\S~1.3]{Che14}. Recall that from the $\Lambda_i$ we can reconstruct the whole
determinant $D$ by means of Amitsur's formula~\cite[\S~1.3]{Che14}.

With a slight abuse of notation we define $\PD_{\bar\rho}(\bar \Z_\ell)$ to be the set of
$r$-dimensional  $\bar\Z_\ell$-valued determinants $D\colon \bar\Z_\ell [G]\to
\bar\Z_\ell$  with  $D\otimes_{\bar\Z_\ell} \bar\F_\ell ={\rm Det} (\bar
  \rho)\otimes_{\F} \bar\F $, which are induced by base change from $\sO'$ to $\bar \Z_\ell$ by $\sO'$-valued continuous determinants $D^{\sO'}: \sO' [G] \to
\sO'$,
where $\sO'\subset\bar\Z_\ell$ runs over all finite   extensions of $\sO$.

\begin{prop}\label{prop:reppseudo}
  \mbox{}
\begin{itemize}
\item[\rm (1)]   The functor $\PD_{\bar \rho}$ is representable  by
$R_{\bar \rho}^P \in \sC$ with universal determinant  $D^{R_{\bar \rho}^P}\colon R_{\bar \rho}^P[G] \to R_{\bar \rho}^P$.
\item[\rm (2)]
  The complete local ring $R^P_{\bar \rho} $ is Noetherian
 and  topologically generated as an $\sO$-algebra by the  finitely many elements
   $\Lambda_j(g_i)$
where $1\le j\le r$ and $g_1,\ldots , g_m\in G$ is a suitable family.
\item[\rm (3)]  If $\bar\rho$ is absolutely irreducible,
$R_{\bar \rho}^P$ coincides with Mazur's universal deformation ring and
$D^{R_{\bar
    \rho}^P}$ is the determinant of the universal deformation.
    \end{itemize}
\end{prop}
We refer to \cite[Proposition~7.59]{Che14} for part~(1), to  \cite[Remark~7.61]{Che14} for part (2),
to  \cite[Example~7.60]{Che14} for part (3).
 Recall  that Mazur's  deformation functor   $\sC \to {\rm Sets}$
assigns to $A\in \sC$ the set of isomorphism classes of continuous representations
$\rho\colon G\to \GL_r(A)$ such that $\rho\otimes_A \F$ is isomorphic to $\bar \rho$, see
for example~\cite[Section~3]{Til96}.
Note that for any $\bar \rho$  we have by definition
\ga{}{ R^P_{\bar \rho} =R^P_{\bar \rho ^{\rm ss}} \notag}
 where
$^{\rm ss}$ indicates the semi-simplification.
We define the {\it universal  deformation space of pseudorepresentations} of $\bar \rho$ by
\ga{}{ \PD_{\bar \rho}= {\rm Spf} R^P_{\bar \rho}. \notag }

\begin{rmk}\label{rmk:pseudobc}
The construction of the universal deformation ring $ R_{\bar \rho}^P $ is compatible with any
finite base change of local rings $\sO\subset \sO'$,  {\it i.e.}\ the universal deformation space of
pseudorepresentations over $\sO'$  with residue field $\F'\subset \bar \F$ a finite
extension of $\F$ and residual condition  $D\otimes_{\sO'} \F'= {\rm Det}(\bar
  \rho)\otimes_{\F} \F'$
is given by  $ R_{\bar \rho}^P \otimes_\sO \sO'$, see~\cite[Proposition~7.59]{Che14}.
  In particular the canonical map
\[
  \Hom_{\sO} (R^P_{\bar \rho} , \bar \Z_\ell )\xrightarrow{\sim} \PD_{\bar\rho} (\bar \Z_\ell)
\]
is bijective.
\end{rmk}

\begin{prop}\label{prop:spseudo}
 Sending a continuous representation $\rho\colon G\to \GL_r(\bar \Q_\ell)$ to its
determinant induces a bijection
\[
{\rm Det} \colon  \sS_{\bar \rho} \xrightarrow{\sim} \PD_{\bar\rho}(\bar \Z_\ell).
\]
\end{prop}

\begin{proof}
By~\cite[Theorem~A]{Che14} there is a bijection between the isomorphism classes of not necessarily continuous
semi-simple representations $\rho\colon G\to\GL_r(\bar \Q_\ell)$ and the not  necessarily continuous
determinants $D\colon \bar \Q_\ell[G] \to  \bar \Q_\ell $.
  We combine this with the simple fact from representation theory that  a semi-simple representation $\rho\colon G\to\GL_r(\bar \Q_\ell)$ is
continuous precisely when its character ${\rm tr}\circ \rho$ has image in a finite
extension of ${\rm Frac}( \sO )$ inside $\bar\Q_\ell$ and is continuous.
\end{proof}

Combining Remark~\ref{rmk:pseudobc} and Proposition~\ref{prop:spseudo} we obtain a canonical
identification
\ga{eq.semsim}{
\sS_{\bar\rho} \xrightarrow{\sim} \Spm ( R^P_{\bar \rho}\otimes_{\sO} \bar\Q_\ell ).
}
We shall see in the next section that~\eqref{eq.semsim} induces the same Zariski topology on
$\sS_{\bar\rho}$ as the one defined in Section~\ref{sec:zartop}.

\section{Characteristic polynomials}\label{sec:charpoly}
In this section $G$ is a profinite group satisfying Mazur's $\ell$-finiteness property.
Recall that for a family  $\underline p= (p_1, \ldots , p_m)$   of $m$ monic  polynomials $p_i$  of degree $r$
over the finite field $\F$, we introduced the complete local deformation ring $R_{\underline p}$ in
Section~\ref{sec:zartop}. When $p_i$ is the  characteristic
polynomial of the matrix $\bar\rho(g_i)$ for a representation
$\bar\rho\colon G\to \GL_r(\F) $ and for $i=1, \ldots, m$,  we also write  $R_{\bar \rho(\underline{g})}$ for $R_{\underline p}$.
Let
\ga{}{ \Def_{\bar \rho(\underline{g})}={\rm Spf} R_{\bar \rho(\underline{g})} \notag}
be the corresponding formal scheme over $\sO$.
We obtain a canonical morphism of formal schemes
\ga{}{ \PD_{\bar \rho}  \xrightarrow{\Ch_{ \underline g}} \Def_{\bar \rho(\underline{g})}\notag}
which sends a pseudorepresentation to the family of associated characteristic polynomials
of $g_1 , \ldots , g_m$.   In view of the identification~\eqref{eq.semsim}
   this induces the map   $\Ch_{ \underline g}\colon  \sS_{\bar
    \rho}\to \Spm (R_{\bar
  \rho(\underline{g})} \otimes_\sO \bar \Q_\ell ) $
from Section~\ref{sec:zartop}.

\begin{lem} \label{lem:fin}  There is a finite family $\underline{\tilde g}\in G^{\tilde m}$
such that for any  finite family $\underline g\in G^m$ containing $\underline{\tilde g}$,
the morphism
 \ga{}{ \Ch_{\underline g} \colon\PD_{\bar \rho}  \to \Def_{\bar \rho(\underline{g})} \notag}  is a closed immersion.
 \end{lem}
\begin{proof}
This is an immediate consequence of
Proposition~\ref{prop:reppseudo}(2).
\end{proof}

\begin{proof}[Proof of Proposition~\ref{prop.Zartop}]
Take $\underline{\tilde g}$ as in Lemma~\ref{lem:fin} and use the identification~\eqref{eq.semsim}.
\end{proof}

\smallskip

Assume given roots  $\underline\mu =(\mu^{(j)}_i)_{\mbox{\scriptsize$\substack{1\le j\le m \\ 1\le i\le r}$}} \in
\mathbb A^{rm}_{\sO}(\F)$ of the $m$ monic polynomials $\underline p$ of degree $r$, we let $\Def_{\underline \mu} $ be the
formal completion of the affine space $\mathbb A^{rm}_{\sO}$ at $\underline\mu$. Sending a point
\[
(\lambda^{(j)}_i)_{i,j} \in \Def_{\underline\mu}(A) \ \text{ to }  \    \left(
(t-\lambda^{(1)}_1) \cdots (t-\lambda^{(1)}_r) , \ldots \right) \in \Def_{\underline p}(A)
\]
induces a finite ``symmetrization'' morphism of $\sO$-formal schemes
\ga{}{ {\rm poly}\colon \Def_{\underline\mu}  \to \Def_{\underline p} . \notag}
For an integer $n$ we write $\underline \mu^n $ for $ ((\mu^{(j)}_i)^n )_{i,j}$. Sending
\[
\underline \lambda  \in \Def_{\underline\mu}(A) \ \text{ to }  \
\underline \lambda^n \in \Def_{\underline \mu^n}(A)
\]
induces a finite morphism of   $\sO$-formal schemes
\[
[n]\colon\Def_{\underline \mu}\to \Def_{\underline\mu^n}.
\]
There exists a unique lower horizontal morphism $[n]$ of $\sO$-formal schemes making the square
\[
  \xymatrix{
    \Def_{\underline \mu} \ar[r]^{[n]} \ar[d]_-{\rm poly}  &  \Def_{\underline\mu^n} \ar[d]^-{\rm poly}\\
   \Def_{\underline p}  \ar[r]_{[n]} &  \Def_{ [n]\underline p}
  }
\]
commutative, where $[n]\underline p$ has the obvious meaning, that is its $j^\mathrm{th}$ component is defined to be
\[(t- (\mu^{(j)}_1)^n) \cdots (t- (\mu^{(j)}_r)^n).\]

\medskip

We now study the compatibility of the universal deformation space of pseudodeformations
with restriction.
Let $U\subset G$ be an open subgroup, $\bar \rho|_U\colon U\to \GL_r(\F)$ be the restriction of $\rho$ to $U$. It induces a morphism
\ga{}{ {\rm rest}\colon {\rm PD}_{\bar \rho} \to {\rm PD}_{\bar \rho|_U} \notag}
by sending $D\colon A[G]\to A$ to its restriction $D|_U\colon A[U]\to A$.

\begin{lem} \label{lem:rest}
The morphism ${\rm rest}$ is finite.
\end{lem}
\begin{proof}
Fix a family $\underline{g}$ in $G$ as in Lemma~\ref{lem:fin}. Choose an integer $n>0$ such that $g_i^n\in U$ for all the $g_i$ of the family and denote by $\underline{g}^n$ the family  $(g_1^n,\ldots, g_m^n)$.
We have a commutative diagram
\ga{}{\xymatrix{ \ar[d]_{\rm rest} {\rm PD}_{\bar \rho}  \ar[r]^{ \Ch_{ \underline{g}  }  } & {\rm D}_{\bar \rho({\underline{g})}} \ar[d]^{[n]}\\
{\rm PD}_{\bar \rho|_U}  \ar[r]_{\Ch_{\underline g^n}} & {\rm D}_{\bar \rho({\underline{g^n})}}
}\notag
}
in which the upper horizontal arrow ${\rm char}_{ \underline{g} }$ is a closed immersion
and the right vertical arrow $[n]$ is finite. This implies that   the left vertical arrow ${\rm rest}$ is finite as well.
\end{proof}

 It is likely that the notion of  induction for pseudorepresentations with respect to an open subgroup $U\subset G$
can be defined and induces  a
finite morphism of universal deformation spaces of pseudorepresentations.
Unfortunately,  this is not documented in the literature.

We now describe  a weak form of induction  which is  sufficient for our purpose.
 For simplicity assume that $U\subset G$ is a  normal subgroup of index $n$.
Let $\bar\rho_U\colon U\to \GL_r(\F )$ be a continuous representation and set $\bar\rho =
\Ind^G_U \bar\rho_U $. Let $\bar\rho|_U\colon U\to \GL_{nr}(\F ) $ be the restriction of
$\bar\rho$ to $U$.  We have
\[
\bar\rho|_U = \bigoplus_{i=1}^n (\bar\rho_U)^{s_i}
\]
where $\Sigma:=\{s_1, \ldots, s_n\} \subset G$ is a set of representatives of $G/U$, and
\ga{}{ (\bar\rho_U)^{s_i} (u)=\bar \rho_U(s_ius_i^{-1}) \ {\rm for} \ u\in U. \notag}
Similarly,
sending a pseudorepresentation $D\colon A[U]\to A$ to the pseudorepresentation
$\oplus_{i=1}^n  D^{s_i}\colon A[U] \to A$ where $D^{s_i}(u)=D(s_ius_i^{-1})$  defines the horizontal morphism $\xi$ in the triangle
\ga{eq.indcompa}{
  \begin{split}
  \xymatrix{
    \PD_{\bar\rho_U}  \ar@{-->}[d]_?\ar[r]^{\xi} & \PD_{\bar\rho|_U} \\
     \PD_{\bar\rho} \ar[ru]_{\rm rest} &
  }
  \end{split}
}
We expect that  \eqref{eq.indcompa} can be extended to a commutative diagram by the indicated dashed
induction arrow.  As $\xi$ is finite by Lemma~\ref{lem:indfinite},   the dashed arrow, if it
exists, is automatically finite.

\begin{lem}\label{lem:indfinite}
The morphism $\xi$ in diagram~\eqref{eq.indcompa} is finite.
\end{lem}

\begin{proof}
 Let $\underline{u}=(u_1,\ldots, u_m), \ u_i\in U$ be a family as in  Lemma~\ref{lem:fin}  for $\bar \rho_U$.
  We consider the family
\[\underline{u}^{\Sigma} =  \left(s_1u_1s_1^{-1},\ldots, s_1u_ms_1^{-1}, s_2u_1s_2^{-1}
 ,\ldots, s_nu_ms_n^{-1} \right) \]
 and the associated roots $\underline \mu$ of the characteristic polynomials of the matrices $\bar\rho
 (\underline u^\Sigma)$.
 We have the diagram
 \ga{}{ \xymatrix{\ar[d]_{\xi} {\rm PD}_{\bar \rho_U} \ar[r]^-{{\rm char}_{\underline{u}^\Sigma}}  &\ar@{-->}[d] {\rm D}_{ \bar \rho_U (\underline{u}^\Sigma)} & \ar[l]_-{\rm poly}  {\rm D}_{\underline \mu} \ar[d]^{\rm id} \\
 {\rm PD}_{\bar \rho|_U}
\ar[r]_-{{\rm char}_{\underline{u}}}   & {\rm D}_{ \bar \rho|_U (\underline{u})}  &
\ar[l]^-{\rm poly} {\rm D}_{\underline \mu} }\notag}
where  we now define the dashed arrow  as the quotient of the identity map ${\rm id}$ as follows.  We label the $\lambda$ coordinates of ${\rm D}_{ \underline{\mu}}$
\ga{}{\left(  \lambda_{i}^{(j)}( a) \ | \ \ a=1,\ldots, n,\ j=1,\ldots, m, \ i=1, \ldots, r \right). \notag}
 The upper poly map is  defined by the elementary symmetric functions on $r$ letters
$ \left( {\rm sym}\left(   \lambda_{i}^{(j)} (a)\right)_{1\le i\le r}  \right)_{a,j}$.  In other words, it is the quotient by the product $\prod_{a,j} \Sigma_r$ where $\Sigma_r$ is the symmetric group in $r$ letters.
The lower ${\rm poly}$ map is defined by
$ \left( {\rm sym}\left(\lambda_{i}^{(j)} (a)\right)_{1\le a\le n,  1\le i\le r}
\right)_{j}$ where ${\rm sym}$ are the  elementary symmetric functions on $rn$ letters. In other words, it is the quotient by the product $\prod_{j} \Sigma_{rn}$.   The embedding $ \prod_a \Sigma_r\subset \Sigma_{rn}$ induces the embedding
 $\prod_{a,j} \Sigma_r\subset \prod_{j} \Sigma_{rn}$, and thus defines the requested dashed arrow
 \ga{}{ {\rm D}_{\underline{\mu}}\ /\ \prod_{a,j} \Sigma_r \dashrightarrow  {\rm D}_{\underline{\mu}}\ /\ \prod_{j} \Sigma_{rn} \notag}
 which is finite.
 This finishes the proof.
\end{proof}

\begin{rmk}\label{rmk:induction}
When evaluated on $\bar\Z_\ell$ the diagram~\eqref{eq.indcompa} becomes commutative if we
define the  dashed arrow on $\bar \Z_\ell$-points as the induction
\[
 \PD_{\bar\rho_U}(\bar \Z_\ell) \cong \sS_{\bar\rho_U} \xrightarrow{\Ind^G_U}
 \sS_{\bar\rho} \cong  \PD_{\bar\rho}(\bar \Z_\ell)
\]
on representations.
Here the isomorphisms are coming from \eqref{eq.semsim} and induction is
understood up to semi-simplification.
\end{rmk}

We now study the case of two-dimensional pseudorepresentations. A two-dimensional   determinant $D\colon A[G]\to A$ has characteristic
polynomial
\[
 g\mapsto   t^2- \tau(g) t + \delta(g) \in A[t],
\]
here for simplicity of notation we write $\tau(g)$ for $\Lambda_1(g) $ and $\delta(g)$ for
$\Lambda_2(g) $.
For elements $g_0,g_1\in G$ we have (see~\cite[Lemma~7.7]{Che14})
\begin{align*}
  \tau(g_0 g_1) &= \tau(g_1 g_0),\\
  \delta(g_0 g_1) &= \delta(g_0) \delta(g_1) ,\\
  \tau(g_0 g_1) &= \tau(g_0) \tau(g_1) -\delta(g_0) \tau(g_0^{-1} g_1 ) .
\end{align*}
From these formulae we deduce that if  $F_2$ is the free group on $2$ elements $g_0, g_1$,
for any element $h$ in $F_2$ there exists $$F\in \mathbb Z
[X_0,X_1,X_2,Y_1,Y_2,Y_1^{-1},Y_2^{-1}]$$
with $\tau(h)=F(\tau(g_0),\tau(g_1), \tau(g_0 g_1) , \delta(g_0) , \delta (g_1))$.
This proves:
\begin{lem}\label{lem:pdefdim2}
If $G$ is topologically generated by $g_0$ and $g_1$ then with
\ga{}{ \underline g=(g_0,g_1,g_0
g_1) \notag}  the morphism  $\Ch_{ \underline g} \colon \PD_{\bar \rho}  \to \Def_{\bar
  \rho(\underline{g})}$ is a closed immersion.
\end{lem}

\section{Compatibility with restriction and induction} \label{sec:resind}

In this section we prove some reductions and compatibilities which enable us to prove Proposition~\ref{prop:lefbelyi}.  So as there  $G$ is a profinite group satisfying Mazur's $\ell$-finiteness property.
As we are interested in the density of the fixed points of powers of an automorphism on a topological space, we formulate the simple Lemma~\ref{lem:densprop} in this context.
 For a topological space $S$ and a homeomorphism $\Phi\colon S\to S,$ we define
 \ga{}{ S^{\Phi^\infty}= \bigcup_{n>0}S^{\Phi^n} \notag}
 and  study
the following density property.

\smallskip

\begin{quote}
${\bf (D)}_{S,\Phi}$:  For any closed subset $Z\subset S$ with
  $\Phi^n(Z)=Z$ for some  integer $n >0 $   the intersection $Z\cap S^{\Phi^\infty}$ is
dense in $Z$.
\end{quote}
If $\Phi$
is clear from the context we omit it in our notation.

\smallskip

Let $\psi^{\sharp}\colon R_2\to R_1$ be a homomorphism of Noetherian Jacobson
rings. Let $\Phi_1\colon R_1\xrightarrow{\sim} R_1$ and $\Phi_2\colon
R_2\xrightarrow{\sim} R_2$ be compatible ring automorphisms. Endow $S_i=\Spm R_i$ with the
Zariski topology ($i=1,2$). Let $\psi\colon S_1\to S_2$ be the induced morphism.

\begin{lem}\label{lem:densprop}\mbox{}
  \begin{itemize}
  \item[\rm (1)] If $\psi$ is surjective then
    \[
      {\bf (D)}_{S_1,\Phi_1} \Rightarrow {\bf (D)}_{S_2,\Phi_2}.
    \]
  \item[\rm (2)] If  the ring homomorphism $\psi^{\sharp}\colon R_2\to R_1$  is finite then
    \[
      {\bf (D)}_{S_2,\Phi_2} \Rightarrow {\bf (D)}_{S_1,\Phi_1}.
    \]
  \end{itemize}
\end{lem}

\begin{proof}
  Part (1) is obvious.
  To show part (2) consider $Z\subset S_1$ closed with $\Phi_1^n(Z)=Z$ for some $n>0$. Then,  replacing $n$ by $m$ for some $m>0$  the
  latter is true for each irreducible component of $Z$, so in order to show that $Z\cap
  S_1^{\Phi_1^\infty}$ is dense in $Z$ we can assume without loss of
  generality that $Z$ is irreducible.

 We assume that the closure $Z'$ of $Z\cap
S_1^{\Phi_1^\infty}$ is not equal to $Z$ and we are going to deduce a contradiction.
Incomparability, see~\cite[Section~V.2.1, Corollary~1]{Bou89}, tells us that we get a proper  inclusion  $\psi(Z') \subsetneq \psi(Z)$
of closed subsets of $S_2$. As the fibres of $\psi$ are finite we have
$\psi^{-1}(S_2^{\Phi_2^\infty})= S_1^{\Phi_1^\infty}$, so $S_2^{\Phi_2^\infty}\cap\psi(Z)\subset
\psi(Z')$. But then $ {\bf (D)}_{S_2,\Phi_2}$ applied to the closed subset $\psi(Z)$ says
that $\psi(Z')=\psi(Z)$, which is a contradiction.
\end{proof}

Let $U\subset G$ be an open subgroup and let $\bar\rho \colon G\to \GL_r(\F )$ be a
continuous representation. Let $\Phi\colon G\to G$ be an automorphism with $\Phi(U)=U$ and
with $\Phi(\bar\rho)\simeq \bar\rho$. We can then deduce compatibility of our density
property with restriction and induction.

\begin{prop} \label{prop:zigzag}\mbox{}
\begin{itemize}
\item[\rm (1)] We have the implication ${\bf (D)}_{\sS_{\bar\rho|_U}} \Rightarrow {\bf
    (D)}_{\sS_{\bar\rho}}$.
  \item[\rm (2)] If $U\subset G$ is normal  and $\bar\rho = \Ind^G_U \bar\rho_U$ with $\Phi(\bar\rho_U) \simeq \bar\rho_U $   we have the implication
  ${\bf (D)}_{\sS_{\bar\rho}} \Rightarrow {\bf (D)}_{\sS_{\bar\rho_U}}$.
\end{itemize}
\end{prop}

\begin{proof}
For part (1) we observe that the restriction map $\sS_{\bar\rho}\to \sS_{\bar\rho|_U}$ is
induced via the identification~\eqref{eq.semsim} by the
finite homomorphism of Noetherian Jacobson rings
\[
  R^P_{\bar\rho|_U}\otimes_\sO \bar\Q_\ell \to R^P_{\bar\rho}\otimes_\sO \bar\Q_\ell .
\]
 The Noetherian Jacobson property of these rings  follows from the general
fact that for a Noetherian complete local $\sO$-algebra $R$ with finite residue field, the
ring $R\otimes_{\sO} \bar\Q_\ell$ is Noetherian and Jacobson, see~\cite[Propositions~A.2.2.2 and~A.2.2.3.(ii)]{GL96}.
The finiteness of the homomorphism is
Lemma~\ref{lem:rest}.
Then part
(1)  follows from Lemma~\ref{lem:densprop}(2) with $S_1=\sS_{\bar\rho}$, $S_2=
\sS_{\bar\rho|_U}$ and $\psi={\rm rest}$.

\smallskip

We prove part (2). In view of Remark~\ref{rmk:induction} taking $\bar\Z_\ell$-points in the
diagram~\eqref{eq.indcompa} we get the commutative diagram
\[
  \xymatrix{
    \sS_{\bar\rho_U}  \ar[d]_{\rm Ind}\ar[r]^{\xi} & \sS_{\bar\rho|_U} \\
     \sS_{\bar\rho} \ar[ru]_{\rm rest} &
  }
\]
By Lemma~\ref{lem:rest} the image $S={\rm rest}( \sS_{\bar \rho}) \subset \sS_{ \bar
  \rho|_U}$  is closed, so it is naturally a maximal spectrum. By
Lemma~\ref{lem:densprop}(1) our assumption  ${\bf (D)}_{\sS_{\bar\rho}}$ implies that
${\bf (D)}_{S}$ holds. As by Lemma~\ref{lem:indfinite} the  map $\xi\colon \sS_{\bar\rho_U} \to S $ is induced  by a
finite ring homomorphism on
maximal spectra, we can apply Lemma~\ref{lem:densprop}(2) and deduce that  ${\bf
  (D)}_{\sS_{\bar\rho_U}}$ holds.
\end{proof}

Let  $X$ be a smooth connected variety over an algebraically closed field $\overline k$. Fix a geometric point
$x\in X(\bar k)$.  We apply the results from the previous sections to  the case  $G=\pi_1^{\rm \acute{e}t}(X,x)$.
As in Section~\ref{sec:conj}, we consider a
continuous representation $\bar\rho\colon G\to \GL_r(\F )$.

 For a morphism of smooth connected varieties  $\iota\colon Y\to X$ and a geometric point
 $y\in Y(\bar k)$ mapping to
 $x$, we let $\iota^*\bar\rho\colon \pi_1^{\rm \acute{e}t}(Y,y)\to \GL_r(\F) $ be the composition
 of $\iota_*\colon \pi_1^{\rm \acute{e}t}(Y,y)\to \pi_1^{\rm \acute{e}t}(X,x)$ with $\bar\rho$.

\begin{prop} \label{prop:curve}
There is a smooth connected one-dimensional $C$ and a locally closed immersion $\iota\colon C\to X$  such that the induced morphism
$\iota^*\colon  \PD_{\bar\rho}\to  \PD_{\iota^* \bar\rho} $ is a closed immersion.
\end{prop}

\begin{proof}
Let $\mathfrak{m} \subset R^P_{\bar \rho}$ be the maximal ideal. By~\cite[Lemma~7.52]{Che14}
there exists an   open normal subgroup $U\subset G$ such that the composed determinant
\[
  R^P_{\bar \rho} [ G ] \xrightarrow{ D^{R^P_{\bar\rho}}}  R^P_{\bar \rho} \to  R^P_{\bar
    \rho} /\mathfrak{m}^2
\]
factors through a determinant $  R^P_{\bar \rho}/\mathfrak{m}^2 [G/U]\to  R^P_{\bar \rho}/\mathfrak{m}^2  $.

It is sufficient to choose $\iota $ such that
\[
\iota^* \colon  R^P_{\iota^* \bar \rho}\to  R^P_{\bar \rho}/\mathfrak{m}^2
\]
is surjective. Recall from Proposition~\ref{prop:reppseudo}(2) that the $\sO$-algebra
$ R^P_{\bar \rho}/\mathfrak{m}^2$ is generated by the coefficients of the characteristic
polynomials $\Lambda_j(\bar g)$ with $1\le j\le r$ and $\bar g\in G/U$. So any closed immersion $\iota\colon Y\to
X$ such that the composition
\[
\pi_1^{\rm \acute{e}t}(Y,y) \xrightarrow{\iota_*} G\to G/U
\]
is surjective will  suffice.

To find such a $\iota$ we can assume without loss of generality that $X\hookrightarrow \A^N_{\overline k}$ is affine and use Bertini's
theorem~\cite[Theorem~6.3]{Jou} applied to the \'etale covering $X'$ of $X$ corresponding to
$U\subset G$ and the unramified map $X'\to \A^N_{\overline k} $. In fact Bertini tells us
that a generic affine
line $L\subset  \A^N_{\overline k}$ has the property that
$X' \times_{ \A^N_{\overline k}}L$ is smooth connected and one-dimensional, so one can
take $C=X \times_{ \A^N_{\overline k}}L$.
\end{proof}

\begin{rmk} \label{rmk:open}
The above argument shows in addition that if $X^\circ \xrightarrow{\iota} X$ is an open embedding, then the induced morphism $\iota^*\colon  \PD_{\bar\rho}\to  \PD_{\iota^* \bar\rho} $ is a closed immersion. Indeed 
$
\pi_1^{\rm \acute{e}t}(X^\circ ,y) \xrightarrow{\iota_*} G
$
is surjective, so a fortiori
\[
\pi_1^{\rm \acute{e}t}(X^\circ ,y) \xrightarrow{\iota_*} G\to G/U
\]
and one argues as above. 
\end{rmk}

\begin{proof}[Proof of Proposition~\ref{prop:lefbelyi}]
 For part (1), we choose $\iota$ as in Proposition~\ref{prop:curve}. Then via the identification~\eqref{eq.semsim}
one sees that $\iota^*\colon \sS_{\bar\rho} \hookrightarrow  \sS_{\iota^* \bar\rho}  $ is a closed
embedding of topological spaces. One can descend $\iota$ to a
morphism $\iota_0\colon Y_0\to X_0$ of varieties over a finite field $k$. Then $\iota^*$
is $\Phi$-equivariant. So ${\bf (D)}_{\sS_{\iota^*\bar\rho}} \Rightarrow {\bf
    (D)}_{\sS_{\bar\rho}}$.

\smallskip

We prove part (2).  By part (1) we may assume that $X_0$ is a smooth geometrically connected curve. Let $\pi\colon Y_0\to X_0$ be a finite \'etale cover trivializing $\bar \rho$. Let $\bar 1$ be the trivial rank $r$ representation on $Y$ with value in $\GL_r(\F)$.
 By  Proposition~\ref{prop:zigzag}(1), we may assume that $\bar \rho=\bar 1$ on the
one-dimensional $X_0$.

Let $X_0\hookrightarrow \bar X_0$ be the normal compactification.   By  \cite[Theorem~5.6]{Sai97} if $p\ge 3$ and \cite[Theorem~1.2]{SY20} if $p=2$,  there is a tame finite Belyi map $\pi\colon \bar X_0\to \P^1$  with ramification in $\{0,1,\infty\}$.  Let $\Sigma\subset \P^1$ be the union of $\pi(\bar X_0
\setminus X_0)$ with  $\{0,1,\infty\}$.  Let $z$ be the coordinate on $\P^1$  with value $0$ at $0$, $1$ at $1$ and $\infty$ at $\infty$.
Let us denote by $a_i$ the $z$ coordinate of  the other closed points of $\Sigma$. Then
$a_i$ lies in the units of a finite field extension of $\F_q$, thus  there is an integer $n>0$ prime to $p$ such that the morphism $z^n\colon \P^1\to \P^1$, which is defined
over $\F_q$, sends $\Sigma$ to $\{0,1,\infty\}$. It follows that the finite morphism
$\tau=z^n \circ \pi\colon \bar X_0 \to \P^1$ has the property that $\bar X_0\setminus
\tau^{-1}(\{0,1,\infty\}) \subset X_0$.  By Remark~\ref{rmk:open}  we can replace $X_0$ by $\bar X_0\setminus
\tau^{-1}(\{0,1,\infty\})$.
Moreover, using  Proposition~\ref{prop:zigzag}(1)   again, we can replace $ X_0$ by the Galois hull of  {\color{blue} $\tau$.}
We finally apply  Proposition~\ref{prop:zigzag}(2) in order to reduce to the case of the curve $\P^1\setminus
\{0,1,\infty\}$ and to the representation $\bar \rho= {\rm Ind}_{\pi_1^{\rm \acute{e}t}(X)}^{\pi_1^{\rm \acute{e}t}(\P^1\setminus
\{0,1,\infty\}) }   \bar 1$, which is tame as $\tau$ is tame. This finishes the proof.
\end{proof}

\section{Proof of Theorem B} \label{sec:thmB}

The aim of this section is to prove Theorem B.
Let $ \mathcal X$ be the scheme $\P^1_W\setminus \{0,1,\infty\}$ over the ring of
Witt vectors $W=W(k)$. Set $K={\rm Frac}(W)$ and fix an algebraic closure $\bar K$ of $K$
together with an embedding $\bar K\hookrightarrow \mathbb C$ and an isomorphism of the
residue field of $\bar K$ with $\bar k$. We also fix a lift $ x^\diamond\in \mathcal
 X(\bar K)$ of our base point  $ x\in
X_0(\bar k)$.

  Fix an orientation for $\mathbb C$  and let $\gamma_0, \gamma_1 ,
  \gamma_\infty\in \pi_1^{\rm top}(\mathcal X(\C),x^\diamond )$ be suitable ``simple'' loops around $0$, $1$
  and $\infty$ such that
  \ga{}{ \gamma_0\cdot \gamma_1\cdot \gamma_\infty=1. \notag}
  Then $ \pi_1^{\rm top}(\mathcal X(\C),x^\diamond )$ is a free group with generators
  $\gamma_0, \gamma_1$.

The \'etale fundamental group $\pi_1^{\rm \acute{e}t}(\mathcal X_{\bar K},x^\diamond) $ is the pro-finite
completion of the topological fundamental group $\pi_1^{\rm top}(\mathcal X(\C),x^\diamond
)$ and there is a canonical outer action of $H={\rm Gal}(\bar K/ K)$ on the former. Let
$\chi\colon  H\to \widehat \Z^\times$ be the cyclotomic character.

\begin{claim}\label{claim:galactloc}
For $h\in {\rm Gal}(\bar K/ K)$ and $a\in \{0,1,\infty\}$ the element $h( \gamma_a )$
(which is well-defined up to conjugation)
is conjugate to $\gamma_a^{\chi(h)}$ in  $\pi_1^{\rm \acute{e}t}(\mathcal X_{\bar K},x^\diamond) $.
\end{claim}

\begin{proof}
  Up to inner automorphisms
one can replace the base point $x^\diamond$ in the \'etale fundamental group by the base
point $y_a\colon \Spec (\bar
K_a)\to \mathcal X_K $, where $\bar K_a$ is an algebraic closure of the fraction field
$K_a$ of $\sO_{\P^1_{\C },a}^h $. As explained in~\cite[\S~1.1.10]{Del73} there is a
corresponding generalized topological base point  of $\mathcal X(\C)$, which we simply
write as $\tilde D^*_a$. Here $D^*_a$ is
a small punctured disk around the point $a\in \P^1_\C(\C)$ and $\tilde D^*_a$ is its
universal covering. Then by
{\it loc. cit.}\ we have a commutative diagram
\[
  \xymatrix{
\mathbb Z  \ar@{=}[r] \ar@{^{(}->}[d] & \pi_1^{\rm top} (D^*_a, \tilde D^*_a) \ar[r]
\ar[d] &  \pi_1^{\rm top} (\mathcal X(\C ),\tilde D^*_a )
\ar[d] \\
\widehat \Z(1)  \ar@{=}[r] & {\rm Gal}(\bar K_a/ K_a ) \ar[r]  &  \pi_1^{\rm \acute{e}t}(\mathcal X_{\bar K} , y_a)
  }
\]
where the vertical maps are pro-finite completions.
\end{proof}

By~\cite[Expos\'e~XIII, \S~2.10 and~4.7]{Gro71} the specialization homomorphism
\[
  {\rm sp} \colon \pi_1^{\rm \acute{e}t}(\mathcal X_{\bar K} ,x^\diamond)\to \pi_1^{\rm t}(X,x)
\]
is surjective and compatible with the action of the Frobenius lift $\Phi$. Here the codomain is the tame fundamental group which we also write
$G^{\rm t}$ in the following.

 We set $g_a={\rm sp}(\gamma_a)$ for $a\in \{0,1,\infty \}$, so we have
 \ga{}{ g_0\cdot g_1\cdot g_\infty=1 \ {\rm in}  \ G^{\rm t}. \notag}
In addition, $G^{\rm t}$ is topologically generated by $g_0, g_1$. By Lemma~\ref{lem:pdefdim2} the map
\ga{eq.charthmB}{
\Ch_{ \underline g} \colon \PD_{\bar \rho}  \to \Def_{\bar \rho(\underline{g})}
}
is a closed immersion, where $\underline g=(g_0,g_1,g_\infty^{-1})$ and where $\PD_{\bar
  \rho}$ classifies pseudorepresentations of the group $G^{\rm t}$. From Claim~\ref{claim:galactloc}  we conclude that
 \ga{}{ \Phi (g_a) \ \text{ is conjugate to }\  g_a^{q}, \notag}
 for $a=0,1,\infty$.
 In particular, as $\bar \rho$ is fixed by $\Phi$, one has
\ga{}{ \Ch( \bar\rho(\underline g^{q}))=\Ch(\bar\rho(\underline g)). \notag}
We assume  without loss of generality,   by replacing $k$ by a finite extension and thus
$\Phi$ by a power,  that the family of roots $\underline \mu$ of the polynomials $ \Ch(\bar\rho(\underline g))$ satisfy $\underline \mu^{q}=\underline
\mu$, thus the isomorphism  $\Def_{\underline \mu} \xrightarrow{[q]}  \Def_{\underline \mu}$ is well defined.
This implies that with the notation as in Section~\ref{sec:charpoly} we obtain a commutative diagram
\[
  \xymatrix{
\PD_{\bar\rho} \ar[r]^{\Phi}  \ar[d]_{\Ch_{ \underline g} }&\PD_{\bar\rho}  \ar[d]^{\Ch_{
    \underline g} }\\
\Def_{\bar \rho(\underline{g})} \ar[r]^{[q]} & \Def_{\bar \rho(\underline{g})}
\\
\Def_{\underline \mu} \ar[r]^{[q]} \ar[u]^{\rm poly} & \Def_{\underline \mu} \ar[u]_{\rm poly}
}
\]
which is $\Phi$-equivariant.

As $\Ch_{ \underline g} $ is a closed immersion we have the implication
\[
  {\bf (D)}_{\Def_{\bar \rho(\underline{g})}(\bar\Z_\ell)}
\Rightarrow {\bf (D)}_{\sS_{\bar\rho}},
\]
see Section~\ref{sec:resind}.

As the morphism $\rm poly$ is surjective on $\bar\Z_\ell$-points,
we have by Lemma~\ref{lem:densprop}(1) the implication
\[
 {\bf (D)}_{\Def_{\underline \mu } (\bar\Z_\ell)} \Rightarrow {\bf (D)}_{\Def_{\bar \rho(\underline{g})}(\bar\Z_\ell)}.
\]

Property $ {\bf (D)}_{\Def_{\underline \mu } (\bar\Z_\ell)}$ is a consequence of~\cite[Theorem~1.7]{EK19} by noting that translating $\Def_{\underline \mu}$
by the Teichm\"uller lift of $\underline\mu$ we can assume that $\underline\mu=(1,\ldots ,
1)$, so that
$\Def_{\underline \mu}(\bar \Z_\ell )$ consists of  the $\bar\Q_\ell$-points of the
multiplicative formal  Lie group associated to $\pi=\Z_\ell^6$.  Indeed, we apply {\it loc. cit.} with $\sigma=[q]$. This implies that the closed subset $Z$ in Property  $ {\bf (D)}_{\Def_{\underline \mu } (\bar\Z_\ell)}$ is a finite union of torsion translates of formal subtori.

 Finally, the points of ${\rm D}_{\underline \mu}$ invariant under $[q]^n$  for  integers $n>0$
  are  precisely the roots of unity.  It follows that those points located on $Z$ are dense on $Z$ by \cite[Lemma~3.1]{EK19}.
 In particular  
a point of $\sS_{\bar \rho}$ is arithmetic if and only if its local monodromies at $0,1,\infty$ are quasi-unipotent.
 This finishes the proof of Theorem B.\qed

\begin{rmk} \label{rmk:akshaymark}
What makes our argument work is the particular property due to Riemann that rank two local systems on
$\P^1_{\C}\setminus \{0,1,\infty\}$ are rigid, see~\cite[p.~1]{Kat96}.
\end{rmk}

\section{Proof of Theorem A} \label{sec:thmA}

The aim of this section is to prove Theorem A.
So $\bar\rho$ is supposed to be absolutely irreducible. The arguments rely on 
~\cite{deJ01},  and are partly similar to \cite{BK}  and~\cite[Section~5]{BHKT19}.

 As recalled in Proposition~\ref{prop:reppseudo}(3)
 $R_{\bar \rho}^P$ is then Mazur's universal  deformation ring, which parametrizes
isomorphism classes of continuous representations $\rho\colon G\to \GL_r(A)$ for $A\in \sC$  such that
$\rho\otimes_A \F$ is isomorphic to $\bar \rho$.
In this case we simply write  $R_{\bar \rho}$ for  $R_{\bar \rho}^P$ and $\Def_{\bar \rho}$ for $\PD_{\bar \rho}$.

 \begin{lem}\label{lem:sm}
The $\sO$-algebra $R_{\bar \rho}$ is formally smooth, {\it i.e.}\ there
   is a non-canonical $\sO$-isomorphism
 \ga{}{ R_{\bar \rho}\cong \sO[[t_1,\ldots, t_b]] .\notag}

 \end{lem}
 \begin{proof}
Let $\Def_{\det\bar\rho} = \Spf R_{\det\bar\rho}$ be the universal deformation space of the degree one representation
\[\det\bar \rho\colon G\to \F^\times.\]
   Let $0\to I\to B\to A\to 0$ be an extension in $\sC$ such that $I \cdot \mathfrak{m}_B=0$ (so the $B$-module structure on $I$ factors through $\F$).
By~\cite[\S~1.6]{Maz89}, there exists a canonical commutative obstruction diagram with exact
rows
\ga{eq.obsseq}{
  \begin{split}
  \xymatrix{
\Def_{\bar \rho} (B) \ar[r]  \ar[d]^{\det} & \ar[r]^-{O}  \Def_{\bar \rho} (A)   \ar[d]^{\det}  & H^2(X ,\sE nd(\sF) )
\ar[d]^{\rm tr} \\
 \Def_{\det\bar \rho} (B) \ar[r]  &  \ar[r]^-{O} \Def_{\det\bar \rho} (A) &  H^2(X ,\F )
}
\end{split}
}
Here $\sF$ is the lisse \'etale sheaf on $X$ corresponding to $\bar \rho$ and we use
that  the canonical map
 \ga{}{H^2(G, {\rm Ad}_{\bar\rho} )  \hookrightarrow H^2(X, \sE nd(\sF)) \notag}
induced by the  Hochschild-Serre spectral sequence of the universal covering $\tilde
X\to X$ is injective. The latter injectivity is due to the fact that $\sF$ is trivialized on   $\tilde
X$ and that its first cohomology on $\tilde X$ vanishes.

\smallskip

{\em First case:}  $X$ is affine.\\
By~\cite[Corollary~3.5]{Art73} we have $ H^2(X ,\sE nd(\sF) )=0$, so the first exact row in~\eqref{eq.obsseq} tells
us that $\Def_{\bar \rho}$ is formally smooth.

\smallskip

{\em Second case:}  $X$ is projective.\\
This case follows from the following two claims and a chase in the
diagram~\eqref{eq.obsseq}.
\end{proof}

\begin{claim}\label{claim:1degsm}
$ \Def_{\det\bar \rho} $ is formally smooth over $\sO$.
\end{claim}

\begin{claim}\label{claim:triso}
The map $\rm tr$ in~\eqref{eq.obsseq} is an isomorphism.
\end{claim}

\begin{proof}[Proof of Claim~\ref{claim:1degsm}]  We know that if $\det \bar\rho =\bar 1$ is trivial, then ${\rm D}_{\bar 1}= \sO \llbracket G^{{\rm ab},\ell} \rrbracket$, where
$G^{{\rm ab},\ell} $ is the abelian, $\ell$-adic \'etale fundamental group of $X$.  In
general, ${\rm D}_{\det\bar\rho}$ is isomorphic to ${\rm D}_{\bar 1}$ by translating with
the Teichm\"uller lift of $\det \bar \rho$,  see~\cite[\S~1.4]{Maz89}.  Thus ${\rm
  D}_{\det\bar\rho}$ is formally smooth, since  $G^{{\rm ab},\ell}$ is torsion free.
\end{proof}

\begin{proof}[Proof of Claim~\ref{claim:triso}]
As the trace map of \'etale sheaves $$\sE nd(\sF)\to \F$$ is surjective and $X$ has
dimension one, the map  $\rm tr$ in~\eqref{eq.obsseq} is surjective as well. So it suffices to show that both $\F$-vector spaces have
dimension one. For $H^2(X,\F)$ this is immediate from Poincar\'e duality.  In order to
apply Poincar\'e duality to  $H^2(X ,\sE
nd(\sF) )$ we recall that the trace pairing induces an isomorphism $\sE nd(\sF)^\vee\cong
\sE nd(\sF) $. So we obtain from duality  an isomorphism   $H^2(X ,\sE
nd(\sF) ) \cong {\rm End}_G(\bar\rho)^\vee = \F$. The  equality comes from the absolute irreducibility of $\bar \rho$ and  Schur's lemma.
\end{proof}

 For an integer $n>0$ we consider the quotient ring $(R_{\bar\rho})_{\Phi^n}=R_{\bar\rho}/I_n $, where $I_n$ is the ideal
generated by $\Phi^n(\alpha)-\alpha$ for all $\alpha\in R_{\bar\rho}$. Based on the presentation
of Lemma~\ref{lem:sm},  we see that $I_n$ is generated by the $b$ elements
\ga{eq.regseqphi}{\Phi^n(t_1)-t_1,
\ldots , \Phi^n(t_b ) -t_b.}
By definition
\ga{}{  \sS_{\bar\rho} \supset  \sA_{\bar \rho}= \bigcup_{n>0} \Spm ( (R_{\bar\rho})_{\Phi^n} \otimes_{\sO} \bar  \Q_\ell  ). \notag}
We  use the following two propositions.
\begin{prop} \label{prop:frob_coinv}
  The ring $(R_{\bar\rho})_{\Phi^n}$ is finite, flat and a complete intersection over
  $\sO$ for any $n>0$.
\end{prop}

\begin{prop} \label{prop:red}
The generic fibre $(R_{\bar \rho})_{\Phi^n}\otimes _{\sO}\bar \Q_\ell $ is reduced for
any $n>0$.
\end{prop}
 Proposition~\ref{prop:red} is the same as Proposition~5.12 in~\cite{BHKT19}.

\begin{proof}[Sketch of proof of Proposition~\ref{prop:frob_coinv}]
As in \cite[\S~3.14]{deJ01},   we have to show that the images of the elements~\eqref{eq.regseqphi} form a
regular sequence in  $R_{\bar \rho} \otimes_{\sO} \F$. The latter is equivalent to
$(R_{\bar \rho})_\Phi \otimes_{\sO} \F$ being zero-dimensional. This is deduced by
verbatim the same argument as {\it loc. cit.}\ in view of the fact that de Jong's conjecture
is known for $\ell>2$ by \cite{Gai07}.
\end{proof}

\begin{proof}[Proof of Proposition~\ref{prop:red}] 
Consider a continuous representation $\rho \colon G\to \GL_r(\sO')$ corresponding to  a
homomorphism $(R_{\bar \rho})_{\Phi^n}\to \sO' $, where $\sO'$ is a discrete valuation ring  which is a
finite extension of $\sO$. Then up to replacing $k$
by a finite extension,
$\rho$ can be extended to a continuous representation
$\rho_0\colon\pi_1^{\rm \acute{e}t}(X_0, x)\to \GL_r(\sO')$, see Remark~\ref{rmk:arithext}.
 As $\bar \rho$ is absolutely irreducible,
$\rho_0\otimes_{\sO'} \bar \Q_\ell$ is irreducible. After a suitable twist we can assume
without loss of generality that
$\det(\rho_0)$ is finite, see~\cite[Proposition~1.3.4]{Del80}.
Then by the Langlands correspondence~\cite[Theorem~VII.6]{Laf02} the lisse sheaf $\sF_0$
corresponding to $\rho_0$ is pure of weight zero.
The tangent space to $\rho$ in $(R_{\bar \rho})_\Phi \otimes_{\sO}\bar\Q_\ell$ is given by
\ga{}{ H^1\left(G, {\rm Ad}_\rho [\frac{1}{\ell}]\right) ^\Phi=H^1\left(X, \sE nd(\sF)[\frac{1}{\ell}]\right)^\Phi=0,\notag}
where the last equality follows from the fact that $H^1(X, \sE nd(\sF
  ))$ has weight one as  $  \sE nd(\sF_0)$ has weight zero. This finishes the proof.
\end{proof}

\begin{proof}[Proof of Theorem A]
  We have to show that an element  $\alpha\in R_{\bar \rho}\otimes_{\sO} \bar\Q_\ell$
  which vanishes on the points $\sA_{\bar\rho}$ is zero.
After replacing $\sO$ by a finite extension we may assume without loss of generality that
$\alpha\in R_{\bar\rho}$.
The vanishing condition means that $\alpha$ is contained in all the maximal
 ideals
corresponding to the points of $\sA_{\bar\rho}$, {\it i.e.}\ that the image of $\alpha$ in the ring
$(R_{\bar\rho})_{\Phi^n}\otimes_{\sO} \bar
\Q_\ell$ is contained in its nilpotent radical for all $n>0$.
As $(R_{\bar\rho})_{\Phi^n}\otimes_{\sO} \bar
\Q_\ell$ is reduced by Proposition~\ref{prop:red}, this means that the image of $\alpha$
vanishes in $(R_{\bar\rho})_{\Phi^n}\otimes_{\sO} \bar
\Q_\ell$ for all $n>0$. By the flatness in Proposition~\ref{prop:frob_coinv} it actually vanishes in
$(R_{\bar\rho})_{\Phi^n}$ for any $n>0$.
By Claim~\ref{claim:intvan} this implies that  $\alpha=0$.
\end{proof}

\begin{claim}\label{claim:intvan}
The canonical map $R_{\bar\rho}\to \lim_n(R_{\bar\rho})_{\Phi^n} $ is injective.
\end{claim}

\begin{proof}
Let $\mathfrak m$ be the maximal ideal of $R_{\bar\rho}$.
For any integer $m >0$, there is an integer $n>0$ such that $\Phi^n$  acts trivially on  $R_{\bar
  \rho}/\mathfrak m^m$ as the latter ring is finite, so $I_n\subset \mathfrak{m}^m$.  Thus
$\cap_{n>0} I_n \subset \cap_{m>0} \mathfrak{m}^m = \{0  \}$.
  \end{proof}

\section{Some Applications}\label{sec:applications}
In this section we make two remarks concerning applications.

\subsection{ The Hard Lefschetz theorem in positive characteristic.} \label{ss:HL}

This application of our Strong Conjecture is motivated by~\cite{Dri}.
Let $f\colon X \to Y$ be a projective morphism of separated schemes of finite type
over an algebraically closed field $\bar k$. Let $\eta\in H^2(X,\Q_\ell)$ be the Chern
class of a relative ample line bundle. Here we omit Tate twists for simplicity of
notation.

One conjectures (see \cite[Remark~1.4]{EK19}) that if  $\sF \in D^b_c(X, \bar \Q_\ell)$ is a
semi-simple perverse sheaf, then the Hard Lefschetz property holds, {\it i.e.}\
 the cup-product
\ga{app.hleq}{\cup\eta^i\colon {}^p H^{-i}f_* \sF \to {}^p  H^{i} f_* \sF }
is an isomorphism for all $i\ge 0$.
It is known that this holds if
\begin{itemize}
\item[(i)] $\sF$ is of geometric origin in the sense of~\cite[\S~6.2.4--6.2.5]{BBD82}
  \item[(ii)] $\bar k$ is the algebraic closure of a finite field $k$ and $f$, $\eta$ and $\sF$
    descend to schemes $X_0,Y_0$ over the field~$k$.
  \end{itemize}

  For part~(ii) one combines~\cite[Theorem~6.2.10]{BBD82} and the Langlands correspondence of
  Drinfeld--Lafforgue~\cite[Theorem~VII.6]{Laf02}. By \cite[Theorem~1.1]{EK19} the Hard Lefschetz property is also known 
  if $\sF$ is a rank $1$ $\bar \Q_\ell$-local system $\sL$, and more generally if it is a twist of such an $\sL$  by a sheaf as in (i) (see Theorem~5.4 in {\it loc.~cit.}).

\begin{prop} \label{prop:conjHL}
If the irreducible constituents of $\sF$ have generic rank at most $r$ and the Strong
Conjecture~\ref{conj:strong} holds for any   representation of degree
$\le r$ then the map~\eqref{app.hleq} is an isomorphism.
\end{prop}

\begin{proof}[Sketch of proof]
  Similar to~\cite[Lemma~6.1.9]{BBD82} one uses a spreading argument in order to reduce to the case
  in which $\bar k$ is the algebraic closure of a finite field $k_0$ and $f$ and $\eta$
  are defined over  $k_0$. Then $\sF$ corresponds to an irreducible representation
  $\rho_{\sF}\colon \pi_1^{\rm \acute{e}t}(U)\to \GL_r(\bar \Q_\ell )$, where $U\subset X$ is a smooth locally closed geometrically irreducible subvariety
   (over which $\sF$ is a shifted  smooth sheaf).

  Let
  $\bar\rho\colon \pi_1^{\rm \acute{e}t}(U)\to \GL_r(\F)$ be the semi-simple reduction of $\rho_\sF$.
  In fact
  each representation $\rho\in \sS_{\bar\rho}$ gives rise via the
  intermediate extension of the associated  smooth sheaf to a $\bar\Q_\ell$-perverse sheaf
  $\sF_\rho$ on $X$.

  Similarly  to~\cite[Corollary~4.3]{EK19} 
  one shows
  \begin{claim}
The subset $Z^\circ \subset \sS_{\bar\rho}$ of those $\rho$ for which the Hard
Lefschetz property for the perverse sheaf $\sF_{\rho}$ fails to hold, is constructible.
\end{claim}

In order to give a complete proof of the claim one would need a theory of perverse
\'etale adic sheaves over fields more general than $\bar\Q_\ell$. Unfortunately, such a
theory does not exist in the literature at the moment, but should be rather formal in
terms of the pro-\'etale topology.

As $Z^\circ$ is also stabilized by the Frobenius $\Phi$, we can apply the Strong
Conjecture to the Zariski closure $Z$ of $Z^\circ$. This implies that $Z^\circ$ contains
an arithmetic point, which contradicts (ii) above.
\end{proof}

\begin{rmk}
Using Proposition~\ref{prop:lefbelyi}, it would be enough in Proposition~\ref{prop:conjHL} to prove the Strong Conjecture in rank $\le r$ on all curves or in any rank on $\P^1\setminus \{0,1,\infty\}$  for a tame $\bar \rho$. 
\end{rmk}

\subsection{Our proof of Theorem~B on $X_0=\P^1\setminus \{0,1,\infty\}$ for $\bar \rho$ tame and $r=2$ implies de Jong's conjecture~\cite[Conjecture~2.3]{deJ01}  in this case} \label{ss:dJ}
Let $\rho_0\colon \pi_1(X_0,x)\to {\GL}_2(\F[[t]])$ be an arithmetic representation and $\rho$ be its restriction to $G$.
By \cite[Proposition~2.4 and Lemma~2.10]{deJ01} we may assume that $\rho\otimes_{\F[[t]]} \F((t)) $ is absolutely  irreducible.
  By the proof of Theorem~B in Section~\ref{sec:thmB}, the $\Phi$-invariant point ${\rm
    Det}(\rho) \in {\rm PD}_{\bar \rho}(\F[[t]])$
  has a  $[q]$-invariant image in $  {\rm D}_{\bar \rho(\underline{g})} (\F[[t]])$. Thus  it lies in
 $  {\rm D}_{\bar \rho(\underline{g})} (\F') \subset  {\rm D}_{\bar \rho(\underline{g})} (\F'[[t]]) $ for a finite extension $\F'\supset \F$. Thus
 $\rho \otimes_{\F[[t]]} \overline{\F((t))}$  comes from a continuous representation
 $G\to {\GL}_2(\F'')$ for a finite extension $\F''\supset \F'$ (\cite[Proposition~2.2]{Bas80}) and thus has finite monodromy.
 This finishes the proof.\qed

\vskip\baselineskip
 We observe that our proof    avoids the use of  the geometric Langlands correspondence,
 which is used in~\cite[Theorem~1.2]{deJ01} to establish the degree two case of de Jong's conjecture.

%%%%%%%%%%%%%%%%%%%%%
% References
%%%%%%%%%%%%%%%%%%%%%

\end{document}